\documentclass[11pt]{article}

\usepackage{epsfig}
\usepackage{amsmath}  % allows stack command for summation
\usepackage{amssymb}

\newcommand{\comment}[1]{}

\def\squarebox#1{\hbox to #1{\hfill\vbox to #1{\vfill}}}
\def\qed{\hspace*{\fill}
        \vbox{\hrule\hbox{\vrule\squarebox{.667em}\vrule}\hrule}\smallskip}
\newenvironment{proof}{\begin{trivlist}
  \item[\hspace{\labelsep}{\em\noindent Proof.~}]
  }{\qed\end{trivlist}}
\newtheorem{lemma}{Lemma}[section]
\newtheorem{theorem}[lemma]{Theorem}
\newtheorem{proposition}[lemma]{Proposition}
\newtheorem{corollary}[lemma]{Corollary}
\newtheorem{claim}[lemma]{Claim}
\newtheorem{observation}[lemma]{Observation}
\newtheorem{definition}[lemma]{Definition}

\def\squareforqed{\hbox{\rlap{$\sqcap$}$\sqcup$}}
\def\qed{\ifmmode\squareforqed\else{\unskip\nobreak\hfil
\penalty50\hskip1em\null\nobreak\hfil\squareforqed
\parfillskip=0pt\finalhyphendemerits=0\endgraf}\fi}

\newcommand{\nats}{\mbox{$\mathbb N$}}
\newcommand{\floor}[1]{\lfloor #1 \rfloor}

\newcommand{\supp}{\mbox{\rm Sp}}

\title{The Complexity of the Evolution of Graph Labelings}

\author{
{\sl Geir Agnarsson}\\
Department of Mathematics\\
George Mason University, MS 3F2\\
4400 University Drive\\
Fairfax, Virginia 22030\\
{\tt geir@math.gmu.edu}
\and
{\sl Raymond Greenlaw}\thanks{Ray gratefully acknowledges
Chiang Mai University for supporting this research.}\\
Department of Computer Science\\
Armstrong Atlantic State University\\
11935 Abercorn Street\\
Savannah, Georgia 31419-1997\\
{\tt raymond.greenlaw@gmail.com}
\and
{\sl Sanpawat Kantabutra}\\
The Theory of Computation Group\\
Computer Science Department\\
Chiang Mai University\\
Chiang Mai, 50200, Thailand\\
{\tt sanpawat@alumni.tufts.edu}}

\date{}

\begin{document}

\maketitle

\begin{abstract}
We study the {\sc Graph Relabeling Problem}---given an undirected,
connected, simple graph $G = (V,E)$, two labelings $L$ and $L'$ of
$G$, and label {\em flip\/} or {\em mutation\/} functions determine
the complexity of transforming or evolving the labeling $L$ into $L'$\@.  
The transformation of $L$ into $L'$ can be viewed as an evolutionary
process governed by the types of flips or mutations allowed.  The
number of applications of the function is the duration of the
evolutionary period.  The labels may reside on the vertices or the
edges. We prove that vertex and edge relabelings have closely related
computational complexities.  Upper and lower bounds on the number of
mutations required to evolve one labeling into another in a general
graph are given.  Exact bounds for the number of mutations required to
evolve paths and stars are given. This corresponds to computing
the exact distance between two vertices in the corresponding {\em Cayley graph}. 
We finally explore both vertex and edge relabeling with {\em privileged labels}, 
and resolve some open problems by providing precise characterizations of when these problems
are solvable.  Many of our results include algorithms for solving the
problems, and in all cases the algorithms are polynomial-time.  The
problems studied have applications in areas such as bioinformatics,
networks, and VLSI.
\end{abstract}

\baselineskip 3.5ex

\section{Introduction}
\label{sec:introduction}

{\em Graph labeling\/} is a well-studied subject in computer science
and mathematics, and a problem that has widespread applications,
including in many other disciplines.  Here we explore a variant of
graph labeling called the {\sc Graph Relabeling Problem} that was
first explored by~Kantabutra~\cite{Ka} and later by the authors of
this paper in~\cite{AgGrKa}. A shorter preliminary version of this
paper appeared in~\cite{SNPD-08}.
Here we present some new results and extend the results given 
in~\cite{Ka,AgGrKa,SNPD-08}.  In particular, we {\em NC\/}$^1$ reduce the
{\sc Vertex Relabeling Problem} to the {\sc Edge Relabeling Problem}
and vice versa, and provide upper and lower bounds on the complexity of
the {\sc Vertex} and {\sc Edge Relabeling Problems}, give tights
bounds on relabeling a path, and provide precise characterizations of
when instances of relabeling with privileged labels are solvable.  The
paper also includes a number of related open problems.

The problem of graph labeling has a rich and long history, and we
recommend Gallian's extensive survey for an introduction to this topic
and for a cataloging of the many different variants of labeling that
have been studied~\cite{gallian2007}.  Puzzles have always intrigued
computer scientists and mathematicians alike, and a number of puzzles
can be viewed as relabeled graphs (for example,
see~\cite{wilson1974}).  One of the most famous of these puzzles is
the so-called {\sc 15-Puzzle}~\cite{slocumsonneveld2006}.  The {\sc
15-Puzzle} consists of 15 tiles numbered from 1 to 15 that are placed
on a $4 \times 4$ board leaving one position empty. The goal is to
reposition the tiles of an arbitrary arrangement into increasing order
from left-to-right and from top-to-bottom by shifting tiles around
while making use of the open hole.  In~\cite{Ka} a generalized version
of this puzzle called the ($n \times n$)-{\sc Puzzle} was used to show
a variant of the {\sc Vertex Relabeling Problem with Privileged
Labels} is {\em NP}-complete.

Graph labeling has been studied in the context of
cartography~\cite{kakoulistollis2001,marksshieber1991}.  And, of
course, there are many special types of labelings which are of great
interest---codings~\cite{Mo}, colorings~\cite{ApHa}, and
rankings~\cite{LaYu} to name but three.  In these cases we are
typically interested in placing labels on the vertices or edges of a
graph in some constrained manner so that certain properties are met.
Such problems are usually not stated in terms of the evolutionary
process that our labeling problems fall under.  In August of~2008
Google searches of 
graph coloring, 
graph labeling, 
graph coding and 
graph ranking 
turned up 
   339,000 hits, 
10,700,000 hits,
 4,060,000 hits and 
 5,640,000 hits respectively.  All of these fields have ongoing research.

The {\sc Graph Relabeling Problem} is not only interesting in its own
right but also has applications in several areas such as
bioinformatics, networks, and VLSI\@.  New applications for such work
are constantly emerging, and sometimes in unexpected contexts.  For
instance, the {\sc Graph Relabeling Problem} can be used to model a
{\em wormhole routing\/} in processor networks in which one-byte
messages called {\em flits}~\cite{wilkinsonallen1999} are sent among
processors. In this example each processor has a limited buffer, one
byte, and the only way to send a message is by exchanging it with
another processor.  Other well-known problems, for example, the {\sc
Pancake Flipping Problem}, can be modeled as a special case of our
problem~\cite{GaPa}.
\begin{figure}
\centering
\centerline{\epsfig{figure=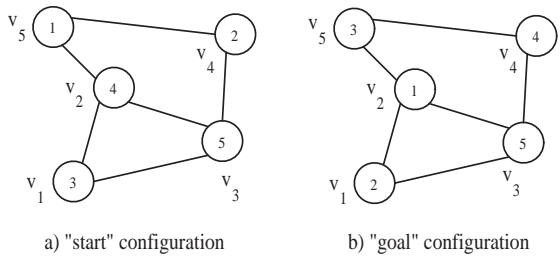,height=100pt,bbllx=0pt,bblly=0pt,bburx=251pt,bbury=118pt}}
\caption{A label relocation problem instance.}
\label{fig:probleminstance}
\end{figure}

This paper is organized as follows: $\S$\ref{sec:definitions} contains
preliminaries and definitions; $\S$\ref{sec:relate} shows the {\sc
Vertex Relabeling Problem} and the {\sc Edge Relabeling Problem} are
{\em NC\/}$^1$ reducible to each other; $\S$\ref{sec:bounds} proves
upper and lower bounds for general graphs for both the {\sc Vertex
Relabeling Problem} and the {\sc Edge Relabeling Problem};
$\S$\ref{sec:exact} contains exact bounds for relabeling a path and a
star; $\S$\ref{sec:intractable} resolves several open problems and
includes results about the {\sc Vertex Relabeling Problem with
Privileged Labels} and the {\sc Edge Relabeling Problem with
Privileged Labels}; $\S$\ref{sec:conclusions} presents concluding
remarks and open problems.  For background material on algorithms we
refer the reader to~\cite{CoLeRiSt}, for graph theory to~\cite{AgGr},
and for basic notations of complexity theory including reducibility
to~\cite{GrHoRu}.

\section{Preliminaries and Problem Definitions}
\label{sec:definitions}

Let $\nats = \{ 1, 2, \ldots \}$ denote the set of the natural numbers.  
Throughout the paper let $G = (V,E)$ be
a simple, undirected, and connected graph.  Let $n = |V|$ and $m =
|E|$; let $V = \{ v_1, \ldots, v_n \}$ and $E = \{ e_1, \ldots, e_m\}$.  
The {\em line graph of\/} $G = (V,E)$ is the graph $L(G) = (E,E')$, 
where $E' = \{ \{e_1, e_2 \} \mid e_1, e_2 \in E \mbox{ and }
e_1 \cap e_2 \neq \emptyset \}$ that is, in the line graph edges
from the original graph become vertices and two of these new vertices
are connected if they share an endpoint in the original graph.  Sometimes
we use $V(G)$ and $E(G)$ to denote the vertices and edges of the graph 
$G$ respectively.

Let $S_V, S_E \subseteq\nats$\@.  A {\em labeling $L_V$ of $V$\/} is a
mapping $L_V: V \mapsto S_V$\@.  A {\em labeling $L_E$ of $E$\/} is a
mapping $L_E: E \mapsto S_E$\@.  In this paper we are usually
interested in $S_V = \{ 1, 2, \ldots, n \}$ and 
$S_E = \{ 1, 2, \ldots, m \}$.  We associate a graph $G$ with labelings using angle
bracket notation, for example, $\langle G, L_V, L_E \rangle$ denotes
the graph $G$ with vertex labeling $L_V$ and edge labeling $L_E$\@.  
A {\em flip\/} or {\em mutation function $f$\/} 
maps triples $\langle G,L_V, L_E \rangle$ to triples $\langle G, L'_V, L'_E \rangle$, or
ordered pairs if we are only interested in one labeling.

We study both vertex and edge mutation functions.  In general, the
mutation function $f$ will be defined based on various properties of
$G$\@.  Here we study just restricted classes of mutation functions.
First, we define the {\em consecutive vertex mutation function},
where $f$ maps a pair $\langle G, L_V \rangle$ to a pair 
$\langle G, L'_V \rangle$, so $(f\circ L_V)(v_i) = f(L_V(v_i)) = L'_V(v_i)$ 
for each $i$, with the following conditions:
\begin{enumerate}
  \item $L_V = L'_V$, except on two vertices $u$ and $w$
  \item $\{ u, w \} \in E$
  \item $L_V(u) = L'_V(w)$ and $L_V(w) = L'_V(u)$
  \item $S_V = \{ 1, 2, \ldots, n \}$
  \item $f$ is a bijection
\end{enumerate}
That is, the labels on the adjacent nodes $u$ and $w$ are swapped,
while all other labels remain the same.  In addition, the set of
labels are chosen from $\{ 1, 2, \ldots, n\}$, and since the
definition requires $f$ to be a bijection, labels are used exactly
once.  It would be interesting to study other types of mutation
functions where, for example, labels along an entire path are mutated,
or where labels can be reused.  One application of the function $f$ is
called a {\em flip} or {\em mutation\/}. We next define a decision
problem that captures the complexity of the evolution of one labeling
into another via the consecutive vertex mutation function.
\begin{definition}{\bf (}{\sc Vertex Relabeling Problem}{\bf)}\\
\label{def:vertexrelabeing}
\hspace*{-.11in}
{\sc Instance:} A graph $G$, labelings $L_V$ and $L'_V$, and $t \in \nats$.\\
{\sc Question:} Can labeling $L_V$ evolve into $L'_V$ in
$t$ or fewer vertex mutations?
\end{definition}

We can similarly define the {\em consecutive edge mutation function},
where $L_E = L'_E$ except for two edges whose labels have been
swapped, and $S_E = \{ 1, 2, \ldots, m \}$.  Note, when employing the
consecutive edge mutation function, the edges whose labels are swapped
must share an endpoint.  We have the following analogous decision
problem for edge relabelings.
\begin{definition}{\bf (}{\sc Edge Relabeling Problem}{\bf)}\\
\label{def:edgerelabeing}
\hspace*{-.11in}
{\sc Instance:} A graph $G$, labelings $L_E$ and $L'_E$,  and $t \in \nats$.\\
{\sc Question:} Can labeling $L_E$ evolve into $L'_E$ in
$t$ or fewer edge mutations?
\end{definition}

In the remainder of the paper we focus on the consecutive versions of
the mutation functions.  The word `consecutive' refers to the fact
that only neighbors can be mutated, that is, labels to be swapped
appear consecutively in the graph.

\section{Relating Vertex and Edge Relabeling}
\label{sec:relate}

The following theorem shows that the computational complexities of the
{\sc Vertex Relabeling Problem} and the {\sc Edge Relabeling Problem}
are closely related.  In the theorem we use {\em NC\/}$^1$ many-one
reducibility---a weak form of reducibility; and therefore, one that
shows a very close relationship between problems---to relate the {\sc
Vertex} and {\sc Edge Relabeling Problems}.
\begin{theorem}{\bf (}{\sc Vertex/Edge Relabeling Related}{\bf)}\\
\label{thm:same}
\hspace*{-.11in}
The {\sc Vertex Relabeling Problem} is NC\/$^1$ many-one reducible to
the {\sc Edge Relabeling Problem}, and the {\sc Edge Relabeling
Problem} is NC\/$^1$ many-one reducible to the {\sc Vertex Relabeling
Problem}.
\end{theorem}
\begin{proof}
We first show that the {\sc Vertex Relabeling Problem} is {\em
NC\/}$^1$ many-one reducible to the {\sc Edge Relabeling Problem}.
Consider an instance $G = (V,E)$, $L_V$ and $L'_V$, and $t \in\nats$
of the {\sc Vertex Relabeling Problem}.  Let $v'_i$, $1 \leq i \leq
n$, be $n$ new vertices.  We construct a corresponding instance of the
{\sc Edge Relabeling Problem} $G' = (V \cup \{ v'_1, \ldots, v'_n \},
E \cup \{ \{ v_i, v'_i \} \mid 1 \leq i \leq n \})$, $L_E$ is such
that $\{ v_i, v'_i \}$ gets $L_V(v_i)$ for $1 \leq i \leq n$ and
$L_E(e_i) = i + m$ for $e_i \in E$, $L'_E$ is such that $\{ v_i, v'_i
\}$ gets $L'_V(v_i)$ for $1 \leq i \leq n$ and $L'_E(e_i) = i + m$ for
$e_i \in E$, and the mutation bound is $3t$.

We now argue the correctness of this reduction.  If we have a {\sc
yes} instance of the {\sc Vertex Relabeling Problem}, then it is clear
that the answer to the resulting instance of the {\sc Edge Relabeling
Problem} is also {\sc yes} since each mutation in $G$ can be mimicked
by three mutations in $G'$\@.  That is, suppose labels $L(v_k)$ and
$L(v_l)$ are mutated.  The following three mutations mimic this swap:
\begin{enumerate}
  \item $L_E( \{ v_k, v'_k \})$ with $L_E(\{ v_k, v_l \})$
  \item new label of $\{v_k, v_l \}$, which is $L_E(\{ v_k, v'_k \})$, with
$L_E( \{ v_l, v'_l \})$
  \item new label of $\{ v_k, v_l \}$, which is $L_E( \{ v_l, v'_l \})$, with new label
of $\{ v_k, v'_k \}$, which is $L_E( \{ v_k, v_l \})$
\end{enumerate}

In the other direction, suppose we have a {\sc yes} instance of the
{\sc Edge Relabeling Problem}.  By construction of $L'_E$ the labels
on the original edges of $G$ remain the same in $L_E$ and $L'_E$\@.
Thus, any movement of labels in $G'$ between the edges $\{ v_k, v'_k
\}$ and $\{ v_l, v'_l \}$, where the edge $\{ v_k, v_l \} \in E'$,
requires a minimum of three mutations to swap the labels on these two
edges and to restore the label on $\{ v_k, v_ l \}$.  Thus, the
corresponding instance of the {\sc Vertex Relabeling Problem} also has
a {\sc yes} answer.

It is not hard to see that if each edge knows its number as part of
the input, then the reduction can be accomplished in {\em NC\/}$^1$
because addition is in the class {\em AC\/}$^0$ which is contained in
{\em NC\/}$^1$.

Now we show that the {\sc Edge Relabeling Problem} is {\em NC\/}$^1$
many-one reducible to the {\sc Vertex Relabeling Problem}.  Consider
an instance $I_E$ of the {\sc Edge Relabeling Problem}, where $G =
(V,E)$, $L_E$ and $L'_E$ are labelings, and $t \in\nats$\@.  We
construct an instance $I_V$ of the {\sc Vertex Relabeling Problem}
using the line graph $L(G) = (E, E')$, $L_{V(L(G))}(e_i) =L_E(e_i)$
for $1 \leq i \leq m$, $L'_{V(L(G))}(e_i) = L'_E(e_i)$ for $1 \leq i
\leq m$, and the bound $t$.

We argue the correctness of the reduction.  Since for each mutation in the
instance $I_E$ of the edges, there is a corresponding mutation of the
vertices in the instance $I_V$, we see that $I_E$ is a {\sc yes}
instance of the {\sc Edge Relabeling Problem} if and only if $I_V$ is
a {\sc yes} instance of the {\sc Vertex Relabeling Problem}.

The reduction can be accomplished in {\em NC\/}$^1$.
This completes the proof of the theorem.
\end{proof}
Notice in the first reduction, we relied on the input being coded so
that each edge ``knows'' its own number.  Without having the input
encoded in some suitable fashion that provides this information, it is
not clear that the reduction is in {\em NC\/}$^1$, however, the
reduction could still be performed in {\em NC\/}$^2$.
Theorem~\ref{thm:same} demonstrates a close relationship between the
{\sc Vertex Relabeling Problem} and the {\sc Edge Relabeling Problem},
when the mutation functions are the consecutive versions.  The theorem
comes in handy when proving results about the {\sc Edge Relabeling
Problem} based on facts about the {\sc Vertex Relabeling Problem}.

\section{Tight Bounds for the Relabeling Problem}
\label{sec:bounds}

This section contains several theorems about the time complexity of
the {\sc Vertex/Edge Relabeling Problems}.  Theorem~\ref{thm:upper}
shows that for an arbitrary graph and two arbitrary labelings at most
$n(n-1)/2$ mutations are required to evolve one vertex labeling into
another.  Corollary~\ref{cor:upper} shows that a similar statement can
be made about the {\sc Edge Relabeling
Problem}. Observation~\ref{obs:lower} yields a lower bound on the
number of mutations required in evolving either vertex-labeled graphs
or edge-labeled graphs.

We begin with the upper bound on the number of flips required to
evolve any given vertex labeling into any other labeling.
\begin{theorem}{\bf (}{\sc Vertex Relabeling Upper Bound}{\bf)}\\
\label{thm:upper}
\hspace*{-.11in}
Let $G = (V,E)$ be a graph, $L_V$ and $L'_V$ vertex labelings, and $t
= n(n-1)/2$, then the answer to the {\sc Vertex Relabeling Problem} is
{\sc yes}.  That is, any labeled graph can evolve into any other
labeled graph in at most $n(n-1)/2$ mutations.
\end{theorem}
\begin{proof}
Let $G = (V,E)$ be any graph.  We need to consider the number of
mutations required to change an arbitrary labeling $L_V$ into an
arbitrary labeling $L'_V$.

We first construct a spanning tree $T$ of $G$\@.  Let $v_1,v_2,\ldots,
v_n$ be the fixed numbering of the vertices (not labels) that denotes
the {\em Pr\"ufer code\/} order when the leaves of $T$ are deleted
during the process of constructing a Pr\"ufer code; note, $v_j \in \{
v_i \mid 1 \leq i \leq n \}$ for $1 \leq j \leq n$.  The Pr\"ufer code
iteratively requires the lowest numbered vertex of degree one to be
removed.  Here we are not interested in the actual Pr\"ufer code
itself but rather just the leaf elimination order given by the
Pr\"ufer code (see~\cite{GrHaPe} for more on the background and
complexity of computing Pr\"ufer codes).

The idea is to transform labels from $L_V$ into their positions in $L'_V$ in
the order specified by the $v_i$'s and along the path in the spanning
tree from their starting position in $L_V$ to their final position in
$L'_V$\@.

Let $\pi$ be a permutation of $\{1,\ldots,n\}$ (presented as
${\pi}_1,\ldots, {\pi}_n$) such that $L_V(v_{{\pi}_i}) = L'_V(v_i)$
for each $i\in\{1,\ldots,n\}$.

To move $L_V(v_{{\pi}_1}) = L'_V(v_1)$ from the initial labeling to
its final position can take at most $n-1$ mutations.  Note, $v_1$ is
an initial leaf in $T$, and $T$ contains exactly $n-1$ edges.

To move $L_V(v_{{\pi}_2}) = L'_V(v_2)$ from the initial labeling to
its final position, we need at most $n-2$ mutations, since $L'_V(v_1)$
is already in its rightful place.

In general, after $i$ iterations, where all of the labels $L'_V(v_1)$
through and including $L'_V(v_i)$ are in their correct places, then,
to move $L_V(v_{{\pi}_{i+1}}) = L'_V(v_{i+1})$ to its correct place,
we need at most $n-i-1$ flips, since the remaining spanning tree
induced by the vertex set, $V(T) - \{ v_{\ell} \mid 1 \leq \ell \leq i
\}$, has exactly $n-i-1$ edges.  Note, we do not perform any flips in
locations of the tree that have already been completed.

All in all, we use at most $(n-1) + (n-2) + \cdots + 1 = n(n-1)/2$
flips to obtain $L'_V$ from $L_V$.
\end{proof}
Note that the proof of Theorem~\ref{thm:upper} is constructive and
provides the sequence of flips to evolve one labeling into
another.  We chose to use the well-known Pr\"{u}er code ordering to
place the labels into leaves first, but any other such leaf ordering
would work as well.  The complexity of the algorithm in
Theorem~\ref{thm:upper} is the complexity of computing a spanning
tree, $\theta(n + m)$, plus the complexity of computing the Pr\"ufer
code elimination order, $\theta(n)$, plus the complexity of the
flips, $\theta(n(n-1)/2)$, which overall is therefore
$\theta(n^2)$.  It is interesting to consider that in the parallel
setting we might be able to compute the sequence of flips
required for the evolution much more quickly than we could actually
execute them sequentially.  We leave this as an open problem.

Corollary~\ref{cor:upper} contains the analogous result to
Theorem~\ref{thm:upper} but for the {\sc Edge Relabeling Problem}.
\begin{corollary}{\bf (}{\sc Edge Relabeling Upper Bound}{\bf)}\\
\label{cor:upper}
\hspace*{-.15in}
Let $G = (V,E)$ be a graph, $L_E$ and $L'_E$ edge labelings, and $t =
m(m-1)/2$, then the answer to the {\sc Edge Relabeling Problem} is
{\sc yes}.  That is, any labeled graph can evolve into any other
labeled graph in at most $m(m-1)/2$ flips.
\end{corollary}
\begin{proof}
The result follows directly from Theorems~\ref{thm:upper}
and~\ref{thm:same}.
\end{proof}

We now discuss the matching lower bounds for the bounds of $t$ given
in Theorem~\ref{thm:upper} and Corollary~\ref{cor:upper}, together
with some well-known folklorish but relevant results.

Consider the path $P_n$ on $n$ vertices. For convenience we represent
a vertex labeling of $P_n$ by a permutation $\pi$ of $\{ 1, 2, \ldots,
n \}$ which we can view as a string $s = {\pi}_1 {\pi}_2 \ldots
{\pi}_n$. For each such string $s$ let $p(s)$ be the number of {\em
inversions} (also known as {\em inversion pairs\/}) of $s$, that is,
$p(s) = |\{\{i,j\} : 1\leq i < j \leq n \mbox{ and } {\pi}_i >
{\pi}_j\}|$.  Note that each mutation reduces or increases the value
of $p(\cdot)$ by exactly one. In other words, if $s'$ is the string
obtained from $s$ by some mutation, then $|p(s') - p(s)| = 1$.  This
well-known observation is stated as a lemma in the original
treatise~\cite[p.~27]{Muir-1882} on determinants.  From this we see
that $p(s)$ is the number of flips or mutations necessary to obtain
${\pi}_1 {\pi}_2 \ldots {\pi}_n$ from $1 \: 2 \ldots
n$~\cite{Muir-1960}. This shows that the bound of
Theorem~\ref{thm:upper} is tight.
\begin{observation}{\bf (}{\sc Lower Bounds for Relabeling Graphs}{\bf)}\\
\label{obs:lower}
\hspace*{-.15in} There is a graph $G = (V,E)$, labelings $L_V$ and
$L'_V$, and $t = (n(n-1)/2) - 1$ such that the {\sc Vertex Relabeling
Problem} has an answer of {\sc no}.  That is, there exist two
labelings that require $n(n-1)/2$ mutations to evolve one into the
other.  There is a graph $H = (V', E')$, labelings $L_{E'}$ and
$L'_{E'}$, and $t = (m(m-1)/2) - 1$ such that the {\sc Edge Relabeling
Problem} has an answer of {\sc no}.
\end{observation}
\begin{proof}
For the permutations $1 \: 2 \ldots n$ and $n \: (n-1) \ldots 1$
(viewed as strings), we clearly have $p(1 \: 2 \ldots n) = 0$ and $p(n
\: (n-1) \ldots 1) = \binom{n}{2} = n(n-1)/2$.  Hence, at least
$n(n-1)/2$ consecutive flips are needed to obtain $n \: (n-1) \ldots
1$ from $1 \: 2 \ldots n$.  The case for edges is similar.
\end{proof}
{\sc Remark:} When we view a labeling of the path $P_n$ on $n$
vertices as a string $s = {\pi}_1 \: {\pi}_2 \ldots {\pi}_n$, we note
that the transformation of $s$ to $1 \: 2 \ldots n$ strongly resembles
standard {\em bubble sort\/}---the simplest of the sorting algorithms on $n$
elements (see~\cite[p.~108]{Knuth3} for discussion and analysis).  In
the case when evolving the string $n \: (n-1) \ldots 1$ to 
$1 \: 2 \ldots n$, the sequence of flips or mutations is precisely the procedure of
bubble sort, except for the very last iteration.

\section{Exact Computations for the Star}
\label{sec:exact}

In this section we determine exactly how many flips are needed to
transform one vertex labeling $L_V$ of $G = (V,E)$ to another vertex
labeling $L'_V$ when $G = K_{1,n-1}$ is the {\em star\/} on $n$
vertices. Considering the cases where $G$ is firstly a path and
secondly a star seems like a good starting point since these
constitute the simplest trees: the path having the largest diameter
(of $n-1$, and smallest maximum degree of two) and the star having the
smallest diameter (of two, and the largest maximum degree of $n-1$).

The case when $G = P_n$, the simple path on $n$ vertices, is a
well-known classic result. Although the statements of these well-known
results for the path are contained in the original work by Thomas
Muir~\cite{Muir-1882} and the expanded and edited
version~\cite{Muir-1960}, the proofs are folklorish or scattered
throughout the literature at best. Hence, in what follows we provide
self-contained proofs of them in our notation.  
Later on these methods for the path will also
be referred to in the case when $G$ is the $n$-star. The case
for the star has also been investigated before in this context, in
particular, in~\cite{Akers-Kris} and from an algorithmic point of view
in~\cite{Misic-thesis} and~\cite{Misic-IEEE}, all nice and interesting
papers on how this applies to connectivity in computer networks.  In
this section we will generalize these results and show how some of
their results follow from ours as special cases.

Consider the transformation of one labeling of the path $P_n$ into
another.  It is clear that the minimum number of mutations needed to
evolve $s = {\pi}_1 \: {\pi}_2 \ldots {\pi}_n$ into $s' = 1 \: 2
\ldots n$ is the same as the minimum number of evolving $s'$ into
$s$. Hence, for the sake of simplicity, we will assume that we are to
evolve $s$ into $s'$. A {\em flip} or {\em mutation sequence\/} $(s_i)_{i=0}^m$ is a
sequence of strings with $s_0 = s$, $s_m = s'$, and where $s_{i+1}$ is
obtained from $s_i$ by a single mutation, $0 \leq i \leq m - 1$.  In
this case we see that for an arbitrary labeling $s = {\pi}_1 \:
{\pi}_2 \ldots {\pi}_n$, we have
\begin{equation}
\label{eqn:path-upper}
p(s) =  |p(s_0) - p(s_m)| = \left| \sum_{i=0}^{m-1}(p(s_i) - p(s_{i+1}))\right|
\leq  \sum_{i=0}^{m-1} |p(s_i) - p(s_{i+1})| = m,
\end{equation}
reestablishing what we know that at least $p(s)$ mutations are needed
to evolve $s$ into $s'$.

By induction on $n$, it is easy to see that $p(s)$ mutations {\em
suffice\/} to evolve $s$ to $s'$: this claim is clearly true for
$n=2$.

Assume that this assertion is true for length $(n-1)$-strings, and let
$s = {\pi}_1 \: {\pi}_2 \ldots {\pi}_n$ be such that $n = {\pi}_i$,
for a fixed $i$, $1\leq i\leq n$.  In this case we have $p(s) = n-i +
p(\hat{s})$, where $\hat{s} = {\pi}_1\ldots
{\pi}_{i-1}{\pi}_{i+1}\ldots {\pi}_n$.  Clearly, in $s$ we can move $n
= {\pi}_i$ to the rightmost position by precisely $n-i$ mutations.  By
induction, we can obtain $1\: 2\ldots (n-1)$ from $\hat{s}$ by
$p(\hat{s})$ mutations. Hence, we are able to evolve $s$ into $s'$
using $p(s)$ mutations.

Finally, we note that if we have two vertex labelings $L_V$ and $L'_V$
of the vertices of the path $P_n$, we can define the corresponding
{\em relative parameter\/} $p(L_V,L'_V)$ as $p(s)$, where $s$ is the
unique permutation obtained from $L_V$ by renaming the labels in
$L'_V$ from left-to-right as $1, 2, \ldots, n$ and reflecting these
new names in $L_V$.  By our previous comment, we have the symmetry
$p(L_V,L'_V) = p(L'_V,L_V)$. This well-known result can now be stated
in our notation as follows.
\begin{observation}{\bf (}{\sc Tight Bound on Path Relabeling Complexity}{\bf)}\\
\label{obs:tightpath}
\hspace*{-.11in}
Let $P_n$ be the path on $n$ vertices, $L_V$ and
$L'_V$ vertex labelings, and $t\in\nats$. Then the answer to the {\sc
Vertex Relabeling Problem} for $P_n$ is {\sc yes} if and only if
$t\geq p(L_V,L'_V)$.
\end{observation}
Finally, note that by Observation~\ref{obs:tightpath} we can always
evolve $L_V$ into $L'_V$ using the minimum of $p(L_V,L'_V)$ mutations,
and repeating the last mutation (or any fixed mutation for that
matter!)  $2k$ times is not going to alter $L'_V$, since repeating a
fixed mutation an even number of times corresponds to the identity (or
neutral) relabeling. Hence, for any nonnegative integer $k$ one can
always evolve $L_V$ into $L'_V$ using $t = p(L_V,L'_V) + 2k$
mutations.

We will now verify that if $L_V$ can evolve into $L'_V$ in $t$
mutations, then $t - p(L_V,L'_V)$ must be even. By renaming the
labels, we may assume $L_V$ is given by the string $s = {\pi}_1 \:
{\pi}_2\ldots {\pi}_n$ and $L'_V$ by the string $s' = 1 \: 2 \ldots
n$.  Now let $(s_i)_{i=0}^m$ and $(s'_i)_{i=0}^{m'}$ be two mutation
sequences with $s_0 = s'_0 = s$ and $s_m = s'_{m'} = s'$.  Since
$p(s_0) = p(s'_0) = p(s)$ and $p(s_m) = p(s'_{m'}) = 0$, we have
\[
p(s) = p(s_0) - p(s_m) = \sum_{i=0}^{m-1}(p(s_i)-p(s_{i+1})) = P_{+} - P_{-},
\]
and
\[
p(s) = p(s'_0) - p(s'_{m'}) = \sum_{i=0}^{m'-1}(p(s'_i)-p(s'_{i+1})) =
P'_{+} - P'_{-},
\]
where
\begin{eqnarray*}
P_{+} & = & |\{i \in\{0,\ldots,m-1\} : p(s_i) - p(s_{i+1}) = 1\}|, \\
P_{-} & = & |\{i \in\{0,\ldots,m-1\} : p(s_i) - p(s_{i+1}) = -1\}|, \\
P'_{+} & = & |\{i \in\{0,\ldots,m'-1\} : p(s'_i) - p(s'_{i+1}) = 1\}|, \mbox{ and}\\
P'_{-} & = & |\{i \in\{0,\ldots,m'-1\} : p(s'_i) - p(s'_{i+1}) = -1\}|.
\end{eqnarray*}
In particular, we have $P'_{+} - P'_{-} = P_{+} - P_{-}$.
Since $m = P_{+} + P_{-}$ and $m' = P'_{+} + P'_{-}$, we obtain
\begin{equation}
\label{eqn:path-parity}
m'-m = (P'_{+} + P'_{-})-(P_{+} + P_{-})=(P'_{+}-P_{+})+(P'_{-}-P_{-}) = 2(P'_{+}-P_{+}),
\end{equation}
and thus $m$ and $m'$ must have the same parity.  This result shows that if
$L_V$ is evolved into $L'_V$ in exactly $t$ mutations, then $t - p(s)$
must be even. This proves the following well-known fact about
permutations, which in our setting reads as follows.
\begin{theorem}[Muir]
\label{thm:path-complete}
Let $P_n$ be the path on $n$ vertices, $L_V$ and $L'_V$ vertex labelings,
and $t\in\nats$. Then we can evolve the labeling $L_V$ into $L'_V$
using $t$ mutations if and only if $t =  p(L_V,L'_V) + 2k$ for some
nonnegative integer $k$.
\end{theorem}
{\sc Remark:} In many places in the literature (especially in books on
abstract algebra), a permutation of $\{1,2,\ldots,n\}$ that swaps two
elements $i\leftrightarrow j$ is called a {\em transposition\/} or a
{\em 2-cycle\/} and is denoted by $(i,j)$.  If $i<j$, then a flip or
mutation in our context is a transposition where $j=i+1$.  In general,
by first moving $j$ to the place of $i$ and then moving $i$ up to the
place of $j$, we see that $(i,j)$ can be obtained by exactly
$2(j-i)-1$ mutations. Since every permutation $\pi$ of $\{1,2,\ldots,n\}$
is a composition of transpositions, say $t$ of
them, then $\pi$ can be obtained from $1\: 2\ldots n$ by $N$
mutations, where $N$ is a sum of $t$ odd numbers.  By
Theorem~\ref{thm:path-complete}, we therefore have that $p(\pi) \equiv
N \equiv t \pmod{2}$. This result gives an alternative and more
quantitative proof of the classic group-theoretic fact that the parity
of the number of transpositions in a composition that yields a given
permutation is unique and only depends on the permutation itself
(see~\cite[p.~48]{Hung} for the classic proof).

\vspace{3 mm}

We now discuss the case $G = K_{1,n-1}$, the star on $n$ vertices. For
our general setup, let $V(G) = \{v_0,v_1,\ldots,v_{n-1}\}$ and $E(G) =
\{\{v_0,v_i\} : i = 1,2,\ldots,n-1\}$, so we assume that $v_0$ is the
center vertex of our star $G$\@. If $L_V$ and $L'_V$ are two vertex
labelings of $G$, we may (by renaming the vertices\/) assume
$L'_V(v_i) = i$ for each $i\in\{0,1,\ldots, n-1\}$. In this case the
initial labeling is given by $L_V(v_i) = \pi(i)$, where $\pi$ is a
permutation of $\{0,1,\ldots,n-1\}$, and so $\pi \in S_n$, the {\em
symmetric group\/} on $n$ symbols $\{0,1,\ldots,n-1\}$ in our case
here.  Call the set of the elements moved by $\pi$ the {\em support\/}
of $\pi$, denote this set by $\supp(\pi)$, and let 
$|\supp(\pi)| = |\pi|$ be its cardinality.  If $\pi$ has the set $S$ as it support,
then we say that $\pi$ is a permutation {\em on\/} $S$ (as supposed to
a permutation of $S$).  Recall that a {\em cycle\/} $\sigma\in S_n$ is
a permutation such that $\sigma(i_{\ell}) = i_{\ell + 1}$ for all
$\ell = 1,\ldots, c-1$, and $\sigma({i_c}) = i_1$, where
$\supp(\sigma) = \{i_1,\ldots,i_c\}\subseteq \{0,1,\ldots,n-1\}$ is
the support of the cycle, so $|\sigma| = c$ here. Such a cycle
$\sigma$ is denoted by $(i_1,\ldots, i_c)$.  Each permutation 
$\pi\in S_n$ is a product of disjoint cycles $\pi =
\sigma_1\sigma_2\cdots\sigma_k$ (see~\cite[p.~47]{Hung}), and this
product/composition is unique. (Note that every two disjoint cycles
commute as compositions of maps $\{0,1,\ldots,n-1\} \rightarrow
\{0,1,\ldots,n-1\}$).  For each permutation $\pi$, denote its number
of disjoint cycles by $\varsigma(\pi)$.  Note that for the star $G$
every mutation or flip has the form $f_i$, where $f_i$ swaps the
labels on $v_0$ and $v_i$ for $i\in\{1,2,\ldots,n-1\}$. Hence, we have
$f_i = (0,i)$, the 2-cycle transposing $0$ and $i$.
\begin{lemma}
\label{lmm:cycle-upper}
Let $G = K_{1,n-1}$ be the star on $n$ vertices.  Let $L_V$ and $L'_V$
be vertex labelings such that $L_V(v_i) = \sigma(i)$ and $L'_V(v_i) =
i$, where $\sigma$ is a cycle with $\supp(\sigma)\subseteq
\{1,2,\ldots,n-1\}$. In this case the labeling $L_V$ can be
transformed into $L'_V$ in $|\sigma| + 1$ or fewer flips.
\end{lemma}
\begin{proof}
If $\sigma = (i_1,\ldots, i_c)$, where $\{i_1,\ldots,i_c\}\subseteq
\{1,2,\ldots,n-1\}$, then apply the composition $f_{\sigma} :=
f_{i_1}f_{i_2}\cdots f_{i_c}f_{i_1}$ to the labeling $L_V$ and obtain
$L'_V$ since
\[
f_{i_1}f_{i_2}\cdots f_{i_c}f_{i_1}\sigma = (0,i_1)(0,i_2)\cdots (0,i_c)(0,i_1)(i_1,\ldots, i_c)
\]
is the identity permutation. Since $f_{\sigma}$ consists of $c+1$
flips altogether, we have the lemma.
\end{proof}
For a cycle $\sigma$ with $\supp(\sigma)\subseteq \{1,2,\ldots,n-1\}$,
let $f_{\sigma}$ denote the composition of the $|\sigma|+1$ label flip
functions as in the previous proof.  By Lemma~\ref{lmm:cycle-upper} we
have the following corollary.
\begin{corollary}
\label{cor:perm-upper}
Let $G = K_{1,n-1}$ be the star on $n$ vertices.  Let $L_V$ and $L'_V$
be vertex labelings such that $L_V(v_i) = \pi(i)$ and $L'_V(v_i) = i$
for $i\in\{0,1,\ldots,n-1\}$ where $\pi(0) = 0$. In this case the
labeling $L_V$ can be transformed into $L'_V$ in
$|\pi|+\varsigma(\pi)$ or fewer flips.
\end{corollary}
\begin{proof}
If $\pi = \sigma_1\cdots\sigma_k$, a product of $k$ disjoint cycles
each having its support in $\{1,2,\ldots, n-1\}$, then apply the
composition $f_{\sigma_k}f_{\sigma_{k-1}}\cdots f_{\sigma_1}$ to the
labeling $L_V$ and obtain $L'_V$.  This composition consists of
$\sum_{i=1}^k(|\sigma_i| + 1) = |\pi| + k = |\pi| + \varsigma(\pi)$
flips altogether.
\end{proof}
Corollary~\ref{cor:perm-upper} establishes an upper bound on how many
flips are needed to transform one labeling into another. This upper
bound is the easier part and coincides with~\cite[Lemma 1,
p.~561]{Akers-Kris}.

We now consider the harder case.  In order to obtain the tight lower
bound, we will define a parameter $q(\cdot)$, a function from the set
of all possible labelings of $G$ into the set of nonnegative integers,
such that each flip either reduces or increases the parameter by
exactly one, just like the number $p(\cdot)$ of inversions of a
permutation on the path.  Before we present the formal definition of
the parameter $q$, we need some notation. For each permutation $\pi$
on $\{0,1,\ldots,n-1\}$, we define a corresponding permutation $\pi^0$
on the same set in the following way:
\begin{enumerate}
  \item If $\pi(0) = 0$, then $\pi^0 := \pi$.
  \item If $\pi(0) = i\neq 0$, then let $j\in \{1,2,\ldots,n-1\}$
        be the unique element with $\pi(j) = 0$.
        In this case we let $\pi^0 := \pi(0,j)$. 
\end{enumerate}
Note that for any permutation $\pi$ on $\{0,1,\ldots,n-1\}$ we always
have $\pi^0(0) = 0$.  If $L_V$ is a vertex labeling of the star $G$
such that $L_V(v_i) = \pi(i)$ for each $i\in\{0,1,\ldots,n-1\}$, then
let $L_V^0$ be the vertex labeling corresponding to the permutation
$\pi^0$, so $L_V^0(v_i) = \pi^0(i)$ for each
$i\in\{0,1,\ldots,n-1\}$. With this preliminary notation we can now
define our parameter.
\begin{definition}
\label{def:q}
Let $L_V : V(G) \rightarrow \{0,1,\ldots,n-1\}$ be a vertex labeling
of the star $G = K_{1,n-1}$ given by $L_V(v_i) = \pi(i)$, where $\pi$
is some permutation of $\{0,1,\ldots,n-1\}$.
\begin{enumerate}
  \item If $\pi(0) = 0$, then let $q(L_V) = |\pi| + \varsigma(\pi)$.
  \item Otherwise, if $\pi(0) = i\neq 0$ and
        hence $\pi(j) = 0$ for some $j$, then let
\[
q(L_V) =
\left\{
\begin{array}{ll}
  q(L_V^0) + 1 \mbox{ if } i = j, \\
  q(L_V^0) - 1 \mbox{ if } i\neq j.
\end{array}
\right.
\]
\end{enumerate}
\end{definition}
Note that $L_V(v_i) = i$ for each $i\in\{0,1,\ldots,n-1\}$ if and only
if $q(L_V) = 0$.

We now want to show that if $L_V$ is a vertex labeling of the star
$G$, and $L'_V$ is obtained from $L_V$ by a single flip, then $|q(L_V)
- q(L'_V)| = 1$.  First we note that if one of the labels swapped by
the single flip is zero, then we either have $L'_V = L_V^0$ or vice
versa $L_V = {L'}_V^0$. Hence, in this case we have directly by
Definition~\ref{def:q} that $|q(L_V) - q(L'_V)| = 1$.

Assume now that neither labels $i$ nor $j$ swapped by the flip is
zero.  In this case we have $L_V(v_0) = i$ and $L'_V(v_0) = j$, and
hence $L_V(v_{\ell}) = L'_V(v_{\ell}) = 0$ for some
$\ell\in\{1,2,\ldots,n-1\}$.  Let the labelings $L_V$ and $L'_V$
on $\{0,1,\ldots,n-1\}$
be given by the permutations $\pi$ and $\pi'$,
respectively.  Since $\pi(k) = j$ and $\pi'(k) = i$ for some $k\neq
\ell$ and $\pi'(\ell) = \pi(\ell) = 0$, we have $\pi' = \pi(0,k)$. 
Using the notation introduced earlier, we have $\pi^0 =
\pi(0,\ell)$ and ${\pi'}^0 = \pi'(0,\ell)$. Since $\pi =
\pi(0,\ell)(0,\ell) = \pi^0(0,\ell)$ and $(0,\ell)(0,k)(0,\ell) =
(k,\ell)$, we have
\begin{equation}
\label{eqn:pio'-pio}
{\pi'}^0 = \pi'(0,\ell) = \pi(0,k)(0,\ell) = \pi^0(0,\ell)(0,k)(0,\ell) = \pi^0(k,\ell).
\end{equation}
Note that (\ref{eqn:pio'-pio}) also implies that
${\pi'}^0(k,\ell) = \pi^0$, and so this observation yields a symmetry
$\pi^0 \leftrightarrow {\pi'}^0$ that we will use later. Also, since
$\pi^0(k) = \pi(k) = j$, $\pi^0(\ell) = \pi(0) = i$, ${\pi'}^0(k) =
\pi'(k) = i$, and ${\pi'}^0(\ell) = \pi'(0) = j$, we see that the
labeling ${L'}_V^0$ is obtained from $L_V^0$ by swapping the labels
$i$ on $v_k$ and $j$ on $v_{\ell}$.

By Definition~\ref{def:q} we have $q(L_V^0) = |\pi^0| + \varsigma(\pi^0)$,
and further by (\ref{eqn:pio'-pio}) we get the following:
\begin{equation}
\label{eqn:L'}
q({L'}_V^0) = |{\pi'}^0| + \varsigma({\pi'}^0) = |\pi^0(k,\ell)| + \varsigma(\pi^0(k,\ell)).
\end{equation}
Note that what happens with the parameter $q$ depends on whether
$\ell\in\{i,j\}$ or not.  Before we consider these cases, we dispatch
with some basic but relevant observations on permutations.
\begin{claim}
\label{clm:cc2}
Let $\sigma_1$ and $\sigma_2$ be two disjoint cycles.  If $i_1\in
\supp(\sigma_1)$ and $i_2\in\supp(\sigma_2)$, then
$\sigma_1\sigma_2(i_1,i_2)$ is a cycle on
$\supp(\sigma_1)\cup\supp(\sigma_2)$.
\end{claim}
\begin{proof}
Let $\sigma_1 = (a_1,\ldots,a_h)$ and $\sigma_2 = (b_1,\ldots,b_k)$,
where $h,k\geq 2$. We may assume that $i_1 = a_1$ and $i_2 = b_1$. In
this case we have $$\sigma_1\sigma_2(i_1,i_2) =
(a_1,\ldots,a_h)(b_1,\ldots,b_k)(a_1,b_1) =
(a_1,b_2,\ldots,b_k,b_1,a_2,\ldots,a_h).$$
\end{proof}
\begin{claim}
\label{clm:c2}
Let $\sigma$ be a cycle and $i_1,i_2\in\supp(\sigma)$ be distinct.
Then one of the following holds for $\sigma(i_1,i_2)$:
\begin{enumerate}
  \item $\supp(\sigma(i_1,i_2)) = \supp(\sigma)$ and $\sigma(i_1,i_2)
    = \sigma_1\sigma_2$---a product of disjoint cycles
    with $\supp(\sigma_1)\cup\supp(\sigma_2)= \supp(\sigma)$.
  \item $\supp(\sigma(i_1,i_2)) = \supp(\sigma)\setminus\{i^*\}$,
    where $i^*\in\{i_1,i_2\}$
    and $\sigma(i_1,i_2)$ is a cycle on $\supp(\sigma)\setminus\{i^*\}$.
  \item $\supp(\sigma(i_1,i_2)) = \emptyset$ and $\sigma = (i_1,i_2)$.
\end{enumerate}
\end{claim}
\begin{proof}
Let $\sigma = (a_1,\ldots, a_h)$, where $h\geq 2$. We may assume
$(i_1,i_2) = (a_1,a_i)$ for some $i\in \{2,\ldots,h\}$. We now
consider the following cases for $h$ and $i$:

If $h=2$, then $i=2$ and $\sigma = (a_1,a_2) = (i_1,i_2)$, and we have
part~3.

If $h\geq 3$ and $i=2$, then $\sigma(i_1,i_2) =
(a_1,\ldots,a_h)(a_1,a_2) = (a_1,a_3,\ldots,a_h)$, and we have part~2.

If $h\geq 3$ and $i=h$, then $\sigma(i_1,i_2) =
(a_1,\ldots,a_h)(a_1,a_h) = (a_2,a_3,\ldots,a_h)$, and again we have
part~2.

Finally, if $h\geq 3$ and $i\not\in\{2,h\}$, then $i\in
\{3,\ldots,h-1\}$ (and hence $h\geq 4$), and $\sigma(i_1,i_2) =
(a_1,\ldots,a_h)(a_1,a_i) = (a_1,a_{i+1},\ldots,a_h)(a_2,\ldots,a_i)$,
and we have part~1.
\end{proof}

We are now ready to consider the cases of whether $\ell\in\{i,j\}$ or
not.

{\sc First case:} $\ell\not\in\{i,j\}$. Directly by definition we have
here that $q(L_V) = q(L_V^0) - 1$ and $q(L'_V) = q({L'}_V^0) - 1$, and
hence $q(L_V) - q(L'_V) = q(L_V^0) - q({L'}_V^0)$.
\begin{proposition}
\label{prp:1st-case}
If $\ell\not\in\{i,j\}$, then $q(L_V) - q(L'_V) = q(L_V^0) -
q({L'}_V^0) = \pm 1$.
\end{proposition}
\begin{proof}
Assuming $\ell\not\in\{i,j\}$, we have $\pi^0(\ell) = i$ and
${\pi'}^0(\ell) = j$, and hence $\ell$ is in the support of both
$\pi^0$ and ${\pi'}^0$.

If $k\not\in\{i,j\}$, then $\{k,\ell\}$ is contained in both
$\supp(\pi^0)$ and $\supp({\pi'}^0)$, and hence by definition we have
that $|\pi^0(k,\ell)| = |\pi^0|$.  Since $\pi^0$ is a product of
disjoint cycles, then, either (i) there are two cycles $\sigma_1$ and
$\sigma_2$ of $\pi^0$ such that $k\in \supp(\sigma_1)$ and $\ell\in
\supp(\sigma_2)$, or (ii) there is one cycle $\sigma$ of $\pi^0$ such
that $\{k,\ell\}\subseteq \supp(\sigma)$.  Since the cycles of $\pi$
commute, we have by Claim~\ref{clm:cc2} in case (i) that
$\varsigma(\pi^0(k,\ell)) = \varsigma(\pi^0) - 1$, and by
Claim~\ref{clm:c2} in case (ii) part 1 that $\varsigma(\pi^0(k,\ell))
= \varsigma(\pi^0) + 1$.  By (\ref{eqn:L'}) this completes the
argument when $k\not\in\{i,j\}$.

If $k\in\{i,j\}$, we may by symmetry ($\pi^0 \leftrightarrow
{\pi'}^0$) assume that $k=i\neq j$.  In this case we have $\pi^0(k) =
j$ so $k\in\supp(\pi^0)$, and ${\pi'}^0(k) = i$ so
$k\not\in\supp({\pi'}^0)$.  Hence, we have $|\pi^0(k,\ell)| = |\pi^0|
- 1$.  Since $\pi^0(\ell) = i = k$, we see that both $k$ and $\ell$
are contained in the same cycle $\sigma$ of $\pi^0$ in its disjoint
cycle decomposition, and they are consecutive.  Moreover, since
$\pi^0(k) = j\neq i$, we see that $|\sigma|\geq 3$.  Again, since
disjoint cycles commute, we have by Claim~\ref{clm:c2} part~2 that
$\varsigma(\pi^0(k,\ell)) = \varsigma(\pi^0)$. By (\ref{eqn:L'}) 
this fact completes the argument when $k=i$, and hence the proof of the proposition.
\end{proof}

{\sc Second case:} $\ell\in\{i,j\}$. By symmetry we may assume $\ell =
i$.  In this case we have directly by definition that $q(L_V) =
q(L_V^0) + 1$ and $q(L'_V) = q({L'}_V^0) - 1$.  Before continuing we
need one more basic observation about permutations.
\begin{claim}
\label{clm:c1}
Let $\sigma$ be a cycle. If $i_1\in\supp(\sigma)$ and
$i_2\not\in\supp(\sigma)$, then $\sigma(i_1,i_2)$ is a cycle on
$\supp(\sigma)\cup\{i_2\}$.
\end{claim}
\begin{proof}
Let $\sigma = (a_1,\ldots, a_h)$. We may assume $(i_1,i_2) = (a_1,b)$,
where $b\not\in\{a_1,\ldots,a_h\}$, and so we get $\sigma(i_1,i_2) =
(a_1,\ldots,a_h)(a_1,b) = (a_1,b,a_2,\ldots,a_h)$.
\end{proof}
\begin{proposition}
\label{prp:2nd-case}
If $\ell = i$, then $q(L_V) - q(L'_V) = q(L_V^0) - q({L'}_V^0) + 2 =
\pm 1$.
\end{proposition}
\begin{proof}
Assuming $\ell = i$, we have $\pi^0(\ell) = i$ and ${\pi'}^0(\ell) =
j$, and hence $\ell\in \supp({\pi'}^0)\setminus \supp(\pi^0)$.

If $k\in\{i,j\}$, then since $k\neq \ell$, we have $k=j$.  Also, since
$\pi^0(k) = j$ and ${\pi'}^0(k) = i$, we have $k\in
\supp({\pi'}^0)\setminus \supp(\pi^0)$. Since $\pi^0$ and ${\pi'}^0$
only differ on $k$ and $\ell$, we have $\supp({\pi'}^0) =
\supp(\pi^0)\cup \{k,\ell\}$, this union being disjoint. From this
fact it is immediate that $|{\pi'}^0| = |\pi^0(k,\ell)| = |\pi^0| + 2$
and $\varsigma(\pi^0(k,\ell)) = \varsigma(\pi^0) + 1$, and hence by
(\ref{eqn:L'}), we have the following:
\[
q({L'}_V^0) = |\pi^0(k,\ell)| + \varsigma(\pi^0(k,\ell)) = |\pi^0| + \varsigma(\pi^0) + 3 = q(L_V^0) + 3,
\]
and hence $q(L_V^0) - q({L'}_V^0) + 2 = q(L_V^0) - (q(L_V^0) + 3) + 2
= -1$, which completes the argument when $k\in\{i,j\}$.

If $k\not\in\{i,j\}$, then since $\pi^0(k) = j$ and ${\pi'}^0(k) = i$,
we have that $k$ is contained in both $\supp(\pi^0)$ and
$\supp({\pi'}^0)$, and therefore $|\pi^0(k,\ell)| = |\pi^0| + 1$.
Since $\pi^0$ is a product of disjoint cycles, there is a unique cycle
$\sigma$ of $\pi^0$ whose support contains $k$.  By
Claim~\ref{clm:c1}, $\sigma(k,\ell)$ is a cycle on
$\supp(\sigma)\cup\{\ell\}$, and hence $\varsigma(\pi^0(k,\ell)) =
\varsigma(\pi^0)$.  By (\ref{eqn:L'}) we therefore have
\[
q({L'}_V^0) = |\pi^0(k,\ell)| + \varsigma(\pi^0(k,\ell)) = |\pi^0| + \varsigma(\pi^0) + 1 = q(L_V^0) + 1,
\]
and hence $q(L_V^0) - q({L'}_V^0) + 2 = q(L_V^0) - (q(L_V^0) + 1) + 2
= 1$, which completes the argument when $k\not\in\{i,j\}$.  This
result completes the proof.
\end{proof}
\begin{corollary}
\label{cor:star-lower}
Let $G = K_{1,n-1}$ be the star on $n$ vertices. If $L_V$ is a vertex
labeling of $G$ and $L'_V$ is a vertex labeling obtained from $L_V$ by a
single flip, then $|q(L_V) - q(L'_V)| = 1$.
\end{corollary}
Corollary~\ref{cor:star-lower} shows that the upper bound given in
Corollary~\ref{cor:perm-upper} is also a lower bound.  We summarize
these results in the following.
\begin{proposition}
\label{prp:star-0}
Let $G = K_{1,n-1}$ be the star on $n$ vertices.  Let $L_V$ and $L'_V$
be vertex labelings such that $L'_V(v_i) = i$ and $L_V(v_i) = \pi(i)$
for $i\in\{0,1,\ldots,n-1\}$, where $\pi(0) = 0$. In this case the
labeling $L_V$ can be transformed into $L'_V$ in $t$ flips if and only
if $t\geq |\pi|+\varsigma(\pi)$.
\end{proposition}
In Proposition~\ref{prp:star-0} we restricted to labelings $L_V$ and
$L'_V$ with $L_V(v_0) = L'_V(v_0) = 0$, which by
Definition~\ref{def:q} is the fundamental case for defining the
parameter $q(L_V)$. Just as we summarized for the case of the path $G
= P_n$ in the beginning of this section, we can likewise define the
{\em relative star parameter\/} $q(L_V,L'_V)$ for any two vertex
labelings $L_V$ and $L'_V$ of the $n$-star $G = K_{1,n-1}$ to be
$q(L''_V)$, where $L''_V$ is the unique vertex labeling obtained from
$L_V$ by renaming the labels of $L'_V$ so that $L'_V(v_i) = i$ for all
$i$.  (Strictly speaking, if $L_V$ and $L'_V$ are given by
permutations $\pi$ and $\pi'$ of $\{0,1,\ldots,n-1\}$, then $L''_V$ is
given by the permutation $\pi'' = \pi({\pi'}^{-1})$.) Clearly, this
relative parameter $q$ is symmetric, $q(L_V,L'_V) = q(L'_V,L_V)$, as
was the case for the path.

As with the path $P_n$, where the parameter $p(\cdot)$ increased or
decreased by exactly one with each mutation or flip, by
Corollary~\ref{cor:star-lower}, so does $q(\cdot)$ for the star $G =
K_{1,n-1}$. Hence, exactly the same arguments used for
(\ref{eqn:path-upper}) and (\ref{eqn:path-parity}) can be
used to obtain the following theorem, our main result of this section.
\begin{theorem}
\label{thm:star-complete}
Let $G = K_{1,n-1}$ be the star on $n$ vertices, $L_V$ and $L'_V$
vertex labelings, and $t\in\nats$. Then we can transform the labeling
$L_V$ into $L'_V$ using $t$ flips if and only if $t = q(L_V,L'_V) +
2k$ for some nonnegative integer $k$, where $q$ is the relative
parameter corresponding to the one in Definition~\ref{def:q}.
\end{theorem}
Theorem~\ref{thm:star-complete} generalizes the results both
from~\cite{Akers-Kris} and~\cite{Misic-IEEE}.

Consider the graph $C$ where its vertex set $V(C)$ consists of all the
$n!$ vertex labelings of the star $G = K_{1,n-1}$, so each vertex
$v_{\pi}$ of $C$ corresponds to a permutation $\pi\in S_n$, and where
two vertices $v_{\pi}$ and $v_{\pi'}$ are connected in $C$ if and only
if $\pi'= \pi(0,i)$ for some $i\in\{1,2,\ldots,n-1\}$. Here $C =
(V(C), E(C))$ is an example of a {\em Cayley graph}, and this
particular one is sometimes ambiguously also referred to as the {\em
star graph\/} in the
literature~\cite[p.~561]{Akers-Kris},~\cite{Misic-thesis},
and~\cite[p.~374]{Misic-IEEE}.  In terms of Cayley graphs, we can
interpret Theorem~\ref{thm:star-complete} as follows:
\begin{corollary}
\label{cor:Cayley}
Let $C$ be the Cayley graph of the $n$-star $G = K_{1,n-1}$.  For any
$\pi, \pi'\in S_n$, let $v_{\pi}, v_{\pi'}\in V(C)$ be the
corresponding vertices of $C$, and $L_V$ and $L'_V$ the corresponding
vertex labelings of $G$.  Then the following holds:
\begin{enumerate}
   \item The distance between $v_{\pi}$ and $v_{\pi'}$
         in $C$ is precisely $q(L_V,L'_V)$.
   \item There is a walk between $v_{\pi}$ and $v_{\pi'}$ in $C$
         of length $d$ if and only if $d = q(L_V,L'_V) + 2k$
         for some nonnegative integer $k$.
\end{enumerate}
\end{corollary}
Other related results regarding the Cayley graph of the star can be
found in~\cite{WSLS} where the distance distribution among the
vertices of the star graph is computed, and in~\cite{QuiMA} where the
cycle structure of the Cayley graph of the star is investigated.

Let $n\in\nats$ be given. Among all permutations $\pi$ on
$\{0,1,\ldots,n-1\}$ with $\pi(0)=0$, clearly a maximum value of
$|\pi|$ is $n-1$, obtained when $\supp(\pi) = \{1,2,\ldots,n-1\}$. Also,
the maximum value of $\varsigma(\pi)$ is $\floor{(n-1)/2}$, obtained
when every cycle of $\pi$ has support of two when $n-1$ is even, or
when every cycle except one (with support of three) has support of two
when $n-1$ is odd.  Hence, among all permutations $\pi$ on
$\{0,1,\ldots,n-1\}$, the maximum value of $|\pi^0| + \varsigma(\pi^0)$
is always $n-1 + \floor{(n-1)/2} = \floor{3(n-1)/2}$.

Consider the star $G = K_{1,n-1}$ and a vertex labeling $L_V$ of $G$
with $q(L_V)$ at maximum. Let $\pi$ be the permutation on
$\{0,1,\ldots,n-1\}$ corresponding to $L_V$, so $L_V(v_i) =
\pi(i)$. If $\pi(0) = 0$, then $\pi = \pi^0$ and by
Definition~\ref{def:q}, the value $q(L_V)$ is at most
$\floor{3(n-1)/2}$.  Assume now that $\pi(0) = i\neq 0$, and hence
$\pi(j) = 0$ for some $j$.  If $i\neq j$, then by
Definition~\ref{def:q} and previous remarks $q(L_V) = q(L_V^0) - 1
\leq \floor{3(n-1)/2} - 1$. Finally if $i = j$, then $q(L_V) =
q(L_V^0)+1$. Since $\pi^0 = \pi(0,j)$, we obtain in this case that
\[
\pi^0(i) = [\pi(0,j)](i) = [\pi(0,i)](i) = i,
\]
and hence $i\not\in\supp(\pi^0)$. Therefore, $|\pi^0|\leq n-2$ and
$\varsigma(\pi^0)\leq \floor{3/2(n-2)}$, and so
\[
q(L_V) = q(L_V^0) + 1 = |\pi^0| + \varsigma(\pi^0) + 1
\leq n-2 + \floor{(n-2)/2} + 1 = \floor{(3n-4)/2}\leq \floor{3(n-1)/2}.
\]
From this inequality we see, in particular, that for a given $n\in\nats$,
the maximum value of $q(L_V)$ among all vertex labelings of the star
on $n$ vertices is $\floor{3(n-1)/2}$.  Hence, we obtain the
next observation as a special case.  This special case was also
observed both in~\cite[p.~561]{Akers-Kris} and
in~\cite[p.~378]{Misic-IEEE}.  In our setting we can state the
following.
\begin{observation}
\label{obs:star-tight}
Let $G = K_{1,n-1}$ be the star on $n$ vertices, $L_V$ and $L'_V$
vertex labelings, and $t = \floor{3(n-1)/2}$, then the answer to the
{\sc Vertex Relabeling Problem} is {\sc yes}.  That is, any labeled
star on $n$ vertices can evolve into any other labeled star in
$t = \floor{3(n-1)/2}$ mutations. Moreover, this value of $t$ is the
smallest possible with this property.
\end{observation}
A group-theoretical interpretation of this result is as follows.
\begin{corollary}
\label{cor:3/2}
Let $T\subseteq S_n$ be a set of $n-1$ transpositions, all of which
move a given element. Then every permutation $\pi\in S_n$ is a
composition of at most $t = \floor{3(n-1)/2}$ transpositions from $T$,
and this value is the least $t$ with this property.
\end{corollary}

We conclude this section by some observations that generalize even
further what we have done for the path and the star, but first we need
some additional notation and basic results.

For $n\in\nats$ let $K_n$ be the complete graph on $n$ vertices, and
let $V(K_n) = \{v_1,\ldots,v_n\}$ be a fixed numbering of the
vertices. Clearly, for each edge $e = \{v_i,v_j\}$ of $K_n$ there is a
corresponding transposition $\tau_e = (i,j)$ in the symmetric group
$S_n$ on $\{1,2,\ldots,n\}$, and vice versa, for each transposition
$\tau = (i,j)\in S_n$ yields an edge $e_{\tau} = \{v_i,v_j\}$ of
$K_n$. This correspondence is 1-1 in the sense that $e_{\tau_e} = e$
and $\tau_{e_{\tau}} = \tau$ for every $e$ and every $\tau$. For edges
$e_1,\ldots, e_m$ of $K_n$ let $G[e_1,\ldots,e_m]$ be the simple graph
induced (or formed) by these edges. In light of
Theorem~\ref{thm:upper} the following observation is clear.
\begin{observation}
\label{obs:spanning-tree}
The transpositions $\tau_1,\ldots,\tau_m \in S_n$ generate the
symmetric group $S_n$ if and only if the graph
$G[e_{\tau_1},\ldots,e_{\tau_m}]$ contains a spanning tree of
$K_n$. In particular, $m\geq n-1$ must hold.
\end{observation}

Consider now a connected simple graph $G = (V(G), E(G))$ on $n$
vertices, where $V(G) = \{v_1,\ldots, v_n\}$ is a fixed numbering.  As
before, a vertex labeling $L_V : V(G) \rightarrow \{1,2,\ldots,n\}$
corresponds to a permutation $\pi$ of $\{1,2,\ldots,n\}$ in $S_n$.
Since $G$ is connected, it contains a spanning tree; and hence, each
vertex labeling $L_V$ of $G$ can be transformed to any other labeling
$L'_V$ of $G$ by a sequence of edge flips or mutations. As for the
path and star, we have in general the following.
\begin{theorem}
\label{thm:general-G}
Let $G$ be a connected simple graph with $V(G) = \{v_1,\ldots, v_n\}$.
For vertex labelings $L_V, L'_V : V(G) \rightarrow \{1,2,\ldots,n\}$
there exists a symmetric nonnegative parameter $p_G(L_V,L'_V)$ and a
function $p_G(n)$ such that we have the following:
\begin{enumerate}
  \item The labeling $L_V$ can be transformed into $L'_V$ in exactly $t$
edge flips if and only if $t = p_G(L_V,L'_V) + 2k$ for some nonnegative
integer $k$.
  \item Every labeling $L_V$ can be transformed into another labeling
$L'_V$ in at most $t$ edge flips if and only if $t\geq p_G(n)$.
\end{enumerate}
\end{theorem}
\begin{proof}
For a given graph $G$ and given vertex labelings $L_V$ and $L'_V$ of
$G$, we define the parameter $p_G(L_V,L'_V)$ as the minimum number of
edge flips needed to transform $L_V$ into $L'_V$. This existence is
guaranteed since every nonempty subset of $\nats\cup \{0\}$ contains a
least element. If $L_V$ can be transformed into $L'_V$ in $t$ edge
flips, then by reversing the process $L'_V$ can be transformed into
$L_V$ in $t$ edge flips as well, so $p_G(L_V,L'_V)$ is clearly
symmetric. By repeating the last edge flip an even number of times, it
is clear that $L_V$ can be transformed into $L'_V$ in $t + 2k$ edge
flips. Assume that $L_V$ can be transformed into $L'_V$ in $t'$ edge
flips. By viewing the edge flips of $t$ and $t'$ as permutations of
$S_n$, they must have the same parity, so $t - t'$ must be even. This
completes the proof of the first part.

By Theorem~\ref{thm:upper} we have that $p_G(L_V,L'_V)\leq n(n-1)/2$
for all vertex labelings $L_V$ and $L'_V$ of $G$\@. Hence, the maximum
of $p_G(L_V,L'_V)$ among all pairs of vertex labelings $L_V$ and
$L'_V$ is also at most $n(n-1)/2$.  Letting $p_G(n)$ be this very
maximum, the second part clearly follows.
\end{proof}
{\sc Remark:} Using the notation of Theorem~\ref{thm:general-G}, what
we have in particular is (i) $p_G(n) \leq n(n-1)/2$ for every
connected graph $G$ on $n$ vertices, (ii) $p_{P_n}(n) = n(n-1)/2$, the
classical result on the number of inversions by Muir~\cite{Muir-1960},
and (iii) $p_{K_{1,n-1}}(n) = \floor{3(n-1)/2}$ for the star.

\section{Relabeling with Privileged Labels}
\label{sec:intractable}

In this section we describe the last variants of the relabeling
problem that we consider in this paper.  We impose an additional
restriction on the flip or mutate operation. Some labels are designated as
{\em privileged}. Our restricted mutations can only take place if 
{\em at least\/} one label of the pair to be mutated is a privileged label.
The problem can be defined for vertices and for edges as follows.

\begin{definition}
{\bf (}{\sc Vertex Relabeling with Privileged Labels Problem}{\bf)}\\
\label{def:privileged}
\hspace*{-.11in}
{\sc Instance:} A graph $G$, labelings $L_V$ and $L'_V$, a nonempty
set $S \subseteq \{ 1, 2, \ldots, n \}$ of privileged labels,
and $t \in\nats$.\\
{\sc Question:} Can labeling $L_V$ evolve into $L'_V$ in
$t$ or fewer restricted vertex mutations?
\end{definition}
\begin{definition}
{\bf (}{\sc Edge Relabeling with Privileged Labels Problem}{\bf)}\\
\label{def:privilegededge}
\hspace*{-.11in}
{\sc Instance:} A graph $G$, labelings $L_E$ and $L'_E$, a nonempty
set $S \subseteq \{ 1, 2, \ldots, m \}$ of privileged labels, and $t \in\nats$.\\
{\sc Question:} Can labeling $L_E$ evolve into $L'_E$ in
$t$ or fewer restricted edge mutations?
\end{definition}
The problems in Definitions~\ref{def:privileged}
and~\ref{def:privilegededge} are increasingly restricted as the
number of privileged labels decreases. Of course, one question is
whether the problems are solvable at all. If $|S|=1$, the {\sc Vertex
Relabeling with Privileged Labels Problem} can be reduced to the {\sc
($n \times n$)-Puzzle Problem}, in which half of the starting
configurations are not solvable~\cite{storey1879}.  This result proved
in~\cite{Ka} shows that the {\sc Vertex Relabeling with Privileged
Labels Problem} is {\em NP}-complete.
\begin{definition}{\bf (}{\sc ($n \times n$)-Puzzle Problem}{\bf )}\\
\label{def:nnpuzzle}
\hspace*{-.11in}
{\sc Instance:} Two $n \times n$ board configurations $B_1$ and
$B_2$, and $k \in \nats$.\\
{\sc Question:} Is there a sequence of at most $k$ moves that
transforms $B_1$ into $B_2$?
\end{definition}
By reducing the {\sc ($n \times n$)-Puzzle Problem} to the {\sc Vertex
Graph Relabeling with Privileged Labels Problem}, by taking $G$ as an
$n \times n$ mesh, $L_V$ corresponding to $B_1$, $L'_V$ corresponding
to $B_2$, $T = \{ n^2 \}$ corresponding to the blank space, and $t = k$,
it is not hard to see that the instance of the
{\sc ($n \times n$)-Puzzle Problem} is ``yes'' if and only if the
answer to the constructed instance of the {\sc Vertex Graph Relabeling with
Privileged Labels Problem} is also ``yes''. We summarize in the following.
\begin{observation}{\bf (}{\sc Intractability, Privileged Labels}{\bf )}
\label{obs:npcomplete}
The {\sc Vertex Graph Relabeling with Privileged Labels Problem} is
NP-complete.
\end{observation}

In Theorem~\ref{thm:same} we
proved that the {\sc Vertex Relabeling Problem} is {\em NC\/}$^1$
many-one reducible to the {\sc Edge Relabeling Problem}, however, that
reduction does not suffice when talking about the versions of the
problems involving privileged labels.  We do not yet know if the {\sc
Edge Relabeling with Privileged Labels Problem} with $|S| = 1$ is {\em
NP}-complete.  It is interesting to note that many other similar games
and puzzles such as the
{\sc Generalized Hex Problem}~\cite{eventarjan1976},
{\sc ($n \times n$)-Checkers Problem}~\cite{fraenkeletal1978},
{\sc ($n \times n$)-Go Problem}~\cite{lichtensteinsipser1978}, and the
{\sc Generalized Geography Problem}~\cite{schaefer1978a} are also {\em NP\/}-complete.

Prior to Observation~\ref{obs:npcomplete} it was still open whether some other unsolvable
instances of the {\sc Vertex/Edge Relabeling with Privileged Labels
Problems} existed. However, we provide some simple examples of
unsolvable instances in this section and provide some interesting
characterizations of both solvable and unsolvable instances of these
problems. We begin with an example.

{\sc Example~A:} Let $n\geq 2$ and consider two vertex labelings $L_V$
and $L'_V$ of the path $P_n$, where we have precisely $k$ privileged
labels $p_1,\ldots,p_k$, where $k\in \{0,1,\ldots, n-2\}$.  For a
fixed horizontal embedding of $P_n$ in the plane, assume the labelings
are given in the following left-to-right order:
\begin{eqnarray*}
L_V  & : & (p_1,\ldots,p_k,1,2,3,\ldots,n-k), \mbox{ and} \\
L'_V & : & (p_1,\ldots,p_k,2,1,3,\ldots,n-k).
\end{eqnarray*}
Note that by any restricted mutation, where one of the labels are
among $\{p_1,\ldots,p_k\}$, the relative left/right order of the
non-privileged labels will remain unchanged. Since the order of the
two non-privileged labels $1$ and $2$ in $L'_V$ is different from the
one of $L_V$, we see that it is impossible to evolve $L_V$ to
$L'_V$ by restricted mutations only. Note that we can push these
labels onto the edges by adding one more edge to the path.  This
example yields the following theorem.
\begin{theorem}{\bf (}{\sc General Insolubility, Privileged Labels}{\bf )}\\
\label{thm:gen-non}
\hspace*{-.11in}
Among all connected vertex labeled graphs on $n$ vertices with $k$
privileged labels where $k\in\{0,1,\ldots,n-2\}$, the {\sc Vertex
Relabeling with Privileged Labels Problem} is, in general, unsolvable.
Among all connected edge labeled graphs on $m$ edges with $k$
privileged labels where $k\in\{0,1,\ldots,m-2\}$, the {\sc Edge
Relabeling with Privileged Labels Problem} is, in general, unsolvable.
\end{theorem}

Note that it is clear that for any connected graph $G$ with all labels but
one being privileged, any mutation is a legitimate transformation, since
for any edge $e = \{u,v\}$ either the label on $u$ or $v$ is
privileged.  Hence, among all connected graphs on $n$ vertices with
$n-1$ privileged labels, the {\sc Vertex Relabeling with Privileged
Labels Problem} is solvable and in $P$\@.  A similar observation holds
for the {\sc Edge Relabeling with Privileged Labels Problem}.

Restricting now to the class of 2-connected simple graphs, we consider
a slight variation of Example~A.

{\sc Example~B:} Let $n\geq 3$ and consider two vertex labelings $L_V$
and $L'_V$ of the cycle $C_n$, where we have precisely $k$ privileged
labels $p_1,\ldots,p_k$, where $k\in \{0,1,\ldots, n-3\}$.  For a
fixed planar embedding of $C_n$, assume the labelings are given
cyclically in clockwise order as follows:
\begin{eqnarray*}
L_V  & : & (p_1,\ldots,p_k,1,2,3,\ldots,n-k), \mbox{ and}\\
L'_V & : & (p_1,\ldots,p_k,2,1,3,\ldots,n-k).
\end{eqnarray*}
Note that by any restricted mutation, where one of the labels are
among $\{p_1,\ldots,p_k\}$, the relative orientation (clockwise or
anti-clockwise) of the non-privileged labels $1$, $2$, and $3$ will
remain unchanged. Since the orientation of $1$, $2$, and $3$ in $L'_V$
is anti-clockwise, and the opposite of the clockwise order of $1$,
$2$, and $3$ in $L_V$, we see again that it is impossible to evolve
$L_V$ to $L'_V$ by restricted mutations.  Notice that we can push the
labels onto the edges.

We summarize the implication of Example~B in the following theorem.
\begin{theorem}{\bf (}{\sc 2-Connected Insolubility, Privileged Labels}{\bf )}\\
\label{thm:2-conn-non}
\hspace*{-.16in}
Among all 2-connected vertex labeled graphs on $n$ vertices with $k$
privileged labels where $k\in\{0,1,\ldots,n-3\}$, the {\sc Vertex
Relabeling with Privileged Labels Problem} is, in general, unsolvable.
Among all 2-connected edge labeled graphs on $m$ edges with $k$
privileged labels where $k\in\{0,1,\ldots,m-3\}$, the {\sc Edge
Relabeling with Privileged Labels Problem} is, in general, unsolvable.
\end{theorem}

We will now fully analyze the case where $G$ is connected and all but
two of the labels are privileged.
\begin{claim}
\label{clm:non-path}
If a simple graph is neither a path nor a cycle, then it has a
spanning tree that is not a path (and hence contains a vertex of
degree at least three).
\end{claim}
\begin{proof}
Let $G$ be a graph that is neither a path nor a cycle.  Then $G$
contains a vertex $u$ of degree greater than or equal to three.
Assigning the weight of one to each edge, we start by choosing three
edges with $u$ as an end-vertex and complete the construction of our
spanning tree using Kruskal's algorithm.
\end{proof}
\begin{claim}
\label{clm:swap}
Among vertex labeled trees, which are not paths, with exactly two
non-privileged labels, any two labels can be swapped using restricted
mutations.
\end{claim}
\begin{proof}
Let $G = (V,E)$ be a tree that is not a path, and $L_V$ a labeling of
the vertices.  For any two distinct vertices $x$ and $y$ denote the
unique path between them by $P(x,y)$.

Assume that we want to swap the labels $L_V(u)$ and $L_V(v)$ on
vertices $u$ and $v$.  We first consider the case where all labels,
except possibly one, on $P(u,v)$, are privileged.  Restricting to
$P(u,v)$, there are $2\partial(u,v)-1$ legitimate mutations that swap
the labels on $u$ and $v$. (Here $\partial(u,v)$ denotes the distance
between $u$ and $v$ in the tree, or the length of $P(u,v)$.  This fact
was noted in the remark right after the proof of
Theorem~\ref{thm:path-complete}.)  Let us denote such a privileged swap by
$SW(u,v)$.

Consider next the case where the labels of $u$ and $v$ are both
non-privileged.  Let $u'$ and $v'$ be vertices such that the
$(u',v')$-path $P^*$ is of maximum length in the tree and such that it
contains $P(u,v)$ as a sub-path. Hence, the three paths $P(u',u)$,
$P(u,v)$, and $P(v,v')$ make up this maximum length path $P^*$. By the
maximality of $P^*$ and our assumption on the tree, there is an
internal vertex $w$ on $P^*$ (note $w\not\in\{u',v'\}$) of degree
three or more, and hence that has a neighbor $w'$ not on $P^*$.  We now
perform the following procedure of legitimate swaps:
\begin{enumerate}
  \item $SW(u,u')$ and $SW(v,v')$,
  \item $SW(u',w')$,
  \item $SW(u',v')$,
  \item $SW(v',w')$, \mbox{ and}
  \item $SW(u,u')$ and $SW(v,v')$.
\end{enumerate}
This procedure has legitimately swapped the labels on $u$ and $v$.

If at least one of the labels of $u$ and $v$ is privileged, but both
of the non-privileged labels do lie on $P(u,v)$, say $x$ and $y$, then
we can perform at least one of the swaps $SW(u,x)$ or $SW(y,v)$, say
$SW(u,x)$, after which we perform the swaps $SW(x,v)$ and $SW(u,x)$ to
complete the legitimate swap.  The case where $SW(y,v)$ was performed
first is handled similarly.  This completes the proof.
\end{proof}

We can now state the following lemma.
\begin{lemma}
\label{lmm:non-path-P}
Among vertex labeled trees, which are not paths, with exactly two
non-privileged labels, the {\sc Vertex Relabeling with Privileged
Labels Problem} is solvable and in $P$.
\end{lemma}
\begin{proof}
Since any transformation from one labeling $L_V$ to another $L'_V$ is a
composition of transpositions, this follows from Claim~\ref{clm:swap}.
\end{proof}
We now have the following summarizing theorem.
\begin{theorem}{\bf (}{\sc Vertex Solubility, Two Privileged Labels}{\bf )}\\
\label{thm:conn-iff}
\hspace*{-.11in}
Among all connected vertex labeled graphs $G$ on
$n\geq 4$ vertices with all but two vertex labels privileged, the {\sc
Vertex Relabeling with Privileged Labels Problem} is solvable if
and only if $G$ is not a path.
\end{theorem}
\begin{proof}
We see from Example~A that for $n\geq 2$ there are labelings of the
vertices of the path $P_n$ that cannot evolve into one another
using restricted mutations.

If $G$ is a cycle on $n\geq 4$ vertices, we can first move the labels
of the non-privileged labels to their desired places by using
appropriate clockwise and/or anti-clockwise sequences of mutations,
and then move all the privileged labels to their places using
mutations as on a path.

If $G$ is neither a path nor a cycle, then by Claim~\ref{clm:non-path}
$G$ has a spanning tree $T$ that is not a path. Restricting to $T$ we
can by Lemma~\ref{lmm:non-path-P} move all the labels to their
desired places within $T$ and hence within $G$.
This completes our proof.
\end{proof}
We obtain the following corollary as a consequence of
Theorems~\ref{thm:conn-iff} and~\ref{thm:same}.
\begin{theorem}{\bf (}{\sc Edge Solubility, Two Privileged Labels}{\bf )}\\
\label{thm:conn-iff-edge}
\hspace*{-.11in}
Among all connected edge labeled graphs $G$ on $n\geq 4$ edges with
all but two edge labels privileged, the {\sc Edge Relabeling
with Privileged Labels Problem} is solvable if and only if $G$ is not
a path.
\end{theorem}

\section{Conclusions and Open Problems}
\label{sec:conclusions}

We have defined several versions of a graph relabeling problem,
including variants involving vertices, edges, and privileged labels,
and proved numerous results about the complexity of these problems,
answering several open problems along the way.  A number of
interesting open problems remain as follows:
\begin{itemize}
  \item Study other types of mutation functions where, for example, labels
along an entire path are mutated, or where labels can be reused.
  \item In the parallel setting, compute the sequence of mutations required
for the transformation of one labeling into another.  The parallel time for
computing the sequence could be much smaller than the sequential time
to execute the mutation sequence.

One result of interest in this direction is the problem of given a
labeled graph, a {\em prescribed flipping sequence}, and two
designated labels $l_1$ and $l_2$ are $l_1$ and $l_2$ flipped?  A
prescribed flipping sequence is an ordering of edges in which each
succeeding edge's labels may be flipped if and only if neither of its
labels has already been flipped.  This problem is {\em
NC}-equivalent to the Lexicographically First Maximal Matching
Problem, and so {\em CC}-complete; see~\cite{GrKaComplexity} for a
list of {\em CC}-complete problems.
  \item For various classes of graphs determine the probability of one
labelings evolving naturally into another.  Such an evolution of a
labeling could be used to model mutation periods.
  \item Study the properties of the graphs of all labelings.  In this graph
all labelings of a given graph are vertices and two vertices are
connected if they are one mutation apart.  Other conditions for edge
placement may also be worthwhile to examine.
  \item Determine if there is a version of the {\sc Edge Relabeling with
Privileged Labels Problem} that is {\em NP\/}-complete.
  \item Define the {\em cost of a mutation sequence\/} to be the sum of the
weights on all edges that are mutated.  Determine mutation sequences
that minimize the cost of evolving one labeling into another.  Explore
other cost functions.
\end{itemize}

\end{document}